\documentclass[12pt]{amsart}
\usepackage{amsmath}
\usepackage{verbatim}
\usepackage{amsfonts}
\usepackage{amssymb}
\usepackage[dvipsnames]{xcolor}
\usepackage[all]{xy}
\usepackage{enumerate}
\usepackage[top=1in, bottom=1in, left=0.9in, right=0.9in]{geometry}
\usepackage{mathrsfs}

\usepackage{stmaryrd}
\usepackage{bbold}

\usepackage{soul}
\usepackage[colorlinks=true,citecolor=PineGreen,linkcolor=RawSienna]{hyperref}
\usepackage{mabliautoref}
\input{kmacros3.sty}

\numberwithin{equation}{section}

\usepackage{mabliautoref}

\renewcommand{\m}{\mathfrak{m}}

\newcommand{\RR}{\mathbb{R}}
\newcommand{\NN}{\mathbb{N}}
\renewcommand{\ZZ}{\mathbb{Z}}
\newcommand{\QQ}{\mathbb{Q}}

\renewcommand{\HK}{\operatorname{HK}}
\newcommand{\e}{\ee_{\HK}}

\newcommand{\ee}{\operatorname{e}}

\renewcommand{\Spec}{\operatorname{Spec}}

\newcommand{\IM}{\operatorname{Im}}

\renewcommand{\Proj}{\operatorname{Proj}}
\renewcommand{\rk}{\operatorname{rk}}

\renewcommand{\leq}{\leqslant}
\renewcommand{\geq}{\geqslant}

\newcommand{\OO}{\mathcal{O}}

\newcommand{\frk}{\mathrm{frk}}

\DeclareMathOperator{\vol}{vol}
\renewcommand{\b}{\mathfrak{b}}

\title{$F$-signature under birational morphisms}
\author{Linquan Ma}
\thanks{Ma was supported in part by NSF Grant DMS \#1836867/1600198.}
\author{Thomas Polstra}
\thanks{Polstra was supported in part by NSF Postdoctoral Research Fellowship DMS $\#1703856$.}
\author{Karl Schwede}
\thanks{Schwede was supported in party by NSF CAREER Grant DMS \#1252860/1501102 and NSF Grant \#1801849.}
\author{Kevin Tucker}
\thanks{Tucker grateful to the NSF for partial support under Grants DMS \#1602070 and \#1707661, and for a fellowship from the Sloan Foundation.}
\address{Department of Mathematics\\ Purdue University\\  West Lafayette\\ IN 47907}
\email{ma326@purdue.edu}
\address{Department of Mathematics\\ University of Utah\\ Salt Lake City\\ UT 84112}
\email{polstra@math.utah.edu}
\address{Department of Mathematics\\ University of Utah\\ Salt Lake City\\ UT 84112}
\email{schwede@math.utah.edu}
\address{Department of Mathematics\\ University of Illinois at Chicago\\Chicago\\  IL 60607}
\email{kftucker@uic.edu}

\begin{document}

\begin{abstract}
We study $F$-signature under proper birational morphisms $\pi : Y \to X$, showing that $F$-signature strictly increases for small morphisms or if $ K_Y \geq \pi ^*K_X$.  In certain cases, we can even show that the $F$-signature of $Y$ is at least twice as that of $X$.  We also provide examples of $F$-signature dropping and Hilbert-Kunz multiplicity increasing under birational maps without these hypotheses.
\end{abstract}

\maketitle

\section{Introduction}

Kunz showed that a local ring $(R, \fram, {\boldsymbol{k}} = {\boldsymbol{k}}^p)$ of positive characteristic is regular if and only if $F^e_* R$ is a free $R$-module, \cite{KunzCharacterizationsOfRegularLocalRings}.  The $F$-signature is a measure of singularities that simply states the percentage of $F^e_* R$ that is free (measured in terms of a rank of a maximal free summand).
$F$-signature was implicitly introduced by K.~Smith and M.~Van Den Bergh \cite{SmithVanDenBerghSimplicityOfDiff} and formally defined by C.~Huneke and G.~Leuschke in \cite{HunekeLeuschkeTwoTheoremsAboutMaximal}, although it wasn't shown to exist until \cite{TuckerFSigExists}.


In this paper we study the behavior of $F$-signature under birational morphisms.  Our main result is as follows.
\begin{mainthm*}[\autoref{thm:secondgoal}, \autoref{thm.CanonicalOrAntiCanonicalBlowup}]
Let $X$ be a strongly $F$-regular variety of dimension $n$ over an algebraically closed field ${\boldsymbol{k}}$ of characteristic $p>0$.  Suppose $\pi \colon Y \to X$ is a proper birational morphism from a normal variety $Y$ and fix a point $y \in \mathrm{Exc}(\pi)$ with $\pi(y)=x$.  Suppose additionally that either:
\begin{enumerate}
    \item $\pi$ is small, \textit{i.e.} $\pi$ is an isomorphism outside a set of codimension at least two in $Y$, or;
    \item The canonical divisor $K_X$ is $\QQ$-Cartier and for every exceptional divisor $E$ containing $y$, we have that $\mathrm{coeff}_E(K_Y-\pi^*K_X) \leq 0$.  For instance, if all discrepancies are non-positive.
\end{enumerate}
Then we have $$s(\OO_{X,x}) < s(\OO_{Y,y}).$$

Furthermore, if $X$ is not Gorenstein at $x$ and $\pi : Y \to X$ is a small morphism obtained as the blowup of either $\O_X(K_X)$ or $\O_X(-K_X)$, then $$2 \cdot s(\OO_{X,x}) \leq s(\OO_{Y,y}).$$
\end{mainthm*}


The first part of our main theorem is a characteristic $p > 0$ analog of a result on normalized volume by Y.~Liu and C.~Xu \cite[Corollary 2.11]{LiuXuKStabilityOfCubicThreefolds}.  We thank both Y.~Liu and C.~Xu for inspiring discussions about the relation between $F$-signature and normalized volume, also see \cite[Theorem 6.14]{LiLiuXuGuidedNormalizedVolume} and \cite{LiuFVolume}.

We finally note that the condition that the blowup of $\O_X(K_X)$ (respectively of $\O_X(-K_X)$) is small can be interpreted as requiring that the graded ring $\O_X \oplus \O_X(K_X) \oplus \O_X(2K_X) \oplus \dots$ is generated in degree $1$ (respectively that $\O_X \oplus \O_X(-K_X) \oplus \O_X(-2K_X) \oplus \dots$ is generated in degree $1$).  This condition is satisfied in surprisingly many rings, including determinantal rings \cite[Corollary~7.10, Theorem~8.8]{BrunsVetterDeterminantal}.

For comparison, in \cite{CaravjalSchwedeTuckerFundamentalGroups}, Javier Carvajal-Rojas and the final two authors of this paper studied the behavior of $F$-signature under finite morphisms (showing that it went strictly up in a controllable way) and used their results to show that the \'etale fundamental group of the punctured spectrum of a strongly $F$-regular singularity was finite.  This was a characteristic $p > 0$ analog of \cite{XuFinitenessOfFundGroups}, and was later shown to imply Xu's result by \cite{BhattGabberOlssonFiniteness}.  Note Xu's proof also used ideas related to volume.

In \autoref{sec.Examples}, we provide examples showing that the $F$-signature can decrease outside of the hypotheses of the main theorem.  We also show that the Hilbert-Kunz multiplicity can increase in that setting as well.
\vskip 9pt
\noindent
{\it Acknowledgements:} We thank Yuchen Liu, Anurag K. Singh, and Chenyang Xu for valuable conversations. We also thank Harold Blum and Takumi Murayama for sharing with us an alternative proof of Lemma~\ref{lem:makeH1growbasic}.

\section{Preliminaries}
\label{sec.Preliminaries}

All schemes and morphisms of schemes considered in this paper will be separated and all rings and schemes will be Noetherian.  Rings and schemes of prime characteristic $p > 0$ will be assumed to be $F$-finite (meaning that the Frobenius map is a finite map).

We are dealing with $F$-signature in this paper and so we recall its definition.  First some notation.  If $R$ is a ring of characteristic $p > 0$ and $M$ is an $R$-module, we use $F^e_* M$ to denote $M$ viewed as an $R$-module under the action of $e$-iterated Frobenius.  For any $R$-module $M$, we use $\frk(M)$ to denote the \emph{free-rank of $M$}, or in other words the maximal rank of a free $R$-module appearing in a direct sum decomposition of $M$, $M = R^{\oplus \frk(M)} \oplus N$.  On the other hand, if $R$ is a domain, we use $\rk(M)$ to denote the (generic) rank of $M$, that is $\rk(M) = \dim_{K(R)} (M \otimes_R K(R))$ where $K(R)$ denotes the fraction field of $R$.

Inspired by the fact that for an $F$-finite local ring $(R, \fram)$, $F^e_* R$ is a free $R$-module if and only if $R$ is regular, we make the following definition which measures how free $F^e_* R$ is, asymptotically.

\begin{definition}[$F$-signature, \cite{HunekeLeuschkeTwoTheoremsAboutMaximal}]
Suppose that $R$ is an $F$-finite domain.  The \emph{$F$-signature of $R$} is defined to be
\[
s(R) = \lim_{e \to \infty} {\frk(F^e_* R) \over \rk(F^e_* R) }.
\]
This limit exists by \cite{TuckerFSigExists} and \cite{DeStefaniPolstraYaoGlobalizing}, also see \cite{PolstraTuckerCombined}.  Furthermore, by \cite[Theorem B]{DeStefaniPolstraYaoGlobalizing}, $\displaystyle s(R) = \min_{\fram \subseteq R} \{ s(R_{\fram}) \}$ where $\fram$ runs over maximal ideals of $R$.  Hence, for any Noetherian integral $F$-finite scheme $X$ we can define
\[
s(X) = \min_{x \in X} s(\O_{X,x}).
\]
\end{definition}

It is clear that $0 \leq s(R) \leq 1$ and it is a fact that $s(R) = 1$ if and only if $R$ is regular by \cite{HunekeLeuschkeTwoTheoremsAboutMaximal} and \cite{DeStefaniPolstraYaoGlobalizing}.  Furthermore, $s(R) > 0$ if and only if $R$ is strongly $F$-regular by \cite{AberbachLeuschke} and \cite{DeStefaniPolstraYaoGlobalizing}.  For our purposes, it will be important to recall that strongly $F$-regular rings are Cohen-Macaulay and normal.

One common tool used to study $F$-signature are \emph{Frobenius degeneracy ideals}.  In particular, if $(R, \fram)$ is an $F$-finite local ring, for each $e > 0$, following \cite{AberbachEnescuStructureOfFPure}, we define
\[
I_e = \{ a \in R \;|\; \phi(F^e_* (aR)) \subseteq \fram, \text{ for all $\phi \in \Hom_R(F^e_* R, R)$ } \}.
\]
It is not difficult to see, \cite{AberbachEnescuStructureOfFPure,YaoObservationsAboutTheFSignature} that
\[
s(R) = \lim_{e \to \infty} {\lambda_R( R/I_e )\over p^{e\dim(R)}}
\]
where $\lambda_R(\bullet)$ denotes the length of the module $\bullet$.  We refer the reader to \cite{HunekeLeuschkeTwoTheoremsAboutMaximal,PolstraLower,PolstraTuckerCombined,TuckerFSigExists} for additional properties of $F$-signature.

Since we are going to study the behavior of $F$-signature under birational maps, we need to understand how maps like $F^e_* \O_X \to \O_X$ (for example, picking out a summand), extend to birational maps.  Suppose $X$ is an $F$-finite normal and integral scheme.  We first notice that $\phi : F^e_* \O_X \to \O_X$ induces a map $F^e_* K(X) \to K(X)$ (simply by tensoring with the fraction field of $X$).  If $\pi : Y \to X$ is a birational map, we obtain an induced map $\widetilde\phi : F^e_* \O_Y \to K(Y)$ (since the fraction fields of $X$ and $Y$ are isomorphic).  It is natural to ask whether
\[
\widetilde\phi(F^e_* \O_Y) \subseteq \O_Y
\]
in which case we say that \emph{$\phi$ extends to a map on $Y$, $\widetilde{\phi} :F^e_* \O_Y \to \O_Y$}.
On the other hand, each $\phi : F^e_* \O_X \to \O_X$ induces a $\bQ$-divisor $\Delta_{\phi} \geq 0$ such that $(1-p^e)(K_X + \Delta_{\phi}) \sim 0$, see \cite[Section 4]{BlickleSchwedeSurveyPMinusE}.

For a proper birational map $\pi: Y \to X$ with $Y$ normal, we may pick canonical divisors $K_Y$ and $K_X$ that agree wherever $\pi$ is an isomorphism.  When working on charts or at stalks of $Y$ and $X$, we continue to use these fixed canonical divisors $K_Y$ and $K_X$.

\begin{lemma}
\label{lem.ExtendingMapsBirationally}
Suppose that $X$ is an $F$-finite normal scheme and that $\pi : Y \to X $ is a finite type birational map from a normal scheme $Y$ with fixed $K_Y$ and $K_X$ as above.  A map $\phi : F^e_* \O_X \to \O_X$ extends to a map $\widetilde\phi : F^e_* \O_Y \to \O_Y$ as above if and only if $K_Y - \pi^*(K_X + \Delta_{\phi}) \leq 0$.

Furthermore, all $\phi : F^e_* \O_X \to \O_X$ extend to $Y$ if either of the following two conditions are satisfied.
\begin{itemize}
\item[(a)]  $\pi$ is small, in other words there is a set $W \subseteq X$ of codimension $\geq 2$ such that $\pi^{-1}(W)$ also has codimension $\geq 2$ in $Y$ and $\pi : Y \setminus \pi^{-1}(W) \to X \setminus W$ is an isomorphism.
\item[(b)]  $K_X$ is $\bQ$-Cartier and $K_Y - \pi^* K_X \leq 0$.
\end{itemize}
\end{lemma}
\begin{proof}
The first statement is \cite[Lemma 7.2.1]{BlickleSchwedeSurveyPMinusE} (notice that $\Delta_{\widetilde{\phi}} \geq 0$ if and only if $\phi$ extends to a map on $Y$).  For (a), notice that $K_Y - \pi^*(K_X + \Delta_{\phi}) = -\pi^{-1}_* \Delta_{\phi} \leq 0$.  Condition (b) is immediate.
\end{proof}
\section{Finitely generated canonical and anti-canonical algebras}
\label{sec.FinitelyGeneratedCanonical}

Before handling the case of more general blowups, we consider the case of a small proper birational map obtained by blowing up either the canonical or anti-canonical local algebra under the special assumption that \emph{those algebras are standard graded}.

\begin{lemma}\textnormal{(\cf \cite[Proposition 3.10]{SannaiOnDualFSignature})}
\label{lem.HowManyOmegasSplit}
Suppose that $(R, \fram, {\boldsymbol{k}})$ is an $F$-finite strongly $F$-regular local ring which is not Gorenstein.  Then we can write
\[
F^e_* R = R^{\oplus a_e} \oplus \omega_R^{\oplus b_e} \oplus M_e
\]
where $M_e$ has no free $R$ or $\omega_R$-summands.  Furthermore $\displaystyle{\lim_{e\to \infty} {b_e\over \rk(F^e_*R)}=s(R)}$ and in particular $\displaystyle{\lim_{e \to \infty} {a_e \over b_e} = 1}$.
\end{lemma}

\begin{proof}
Consider a split surjection $F^e_*R\to \omega_R^{\oplus b_e}$. Then the induced map $ \Hom_R(\omega_R^{\oplus b_e}, \omega_R)\to \Hom_R(F^e_*R,\omega_R)$ remains split. Moreover, $\Hom_R(F^e_*R, \omega_R)\cong F^e_*\Hom_R(R,\omega_R)\cong F^e_*\omega_R$ and $\Hom_R(\omega_R^{\oplus b_e},\omega_R)\cong R^{\oplus b_e}$. Therefore $b_e$ is no more than $\frk(F^e_*\omega_R)$. Conversely, if we set $c_e=\frk(F^e_*\omega_R)$ and consider a split surjective map $F^e_*\omega_R\to R^{\oplus c_e}$ then the induced map $\Hom_R(R^{\oplus c_e}, \omega_R)\to \Hom_R(F^e_*\omega_R, \omega_R)$ remains split. Moreover, $\Hom_R(R^{\oplus c_e}, \omega_R)\cong R^{\oplus c_e}$ and $\Hom_R(F^e_*\omega_R, \omega_R)\cong F^e_*\Hom_R(\omega_R,\omega_R)\cong F^e_*R$. Therefore $c_e=\frk(F^e_*\omega_R)$ is no more than $b_e$ and so the two numbers coincide. In particular, $b_e=\frk(F^e_*\omega_R)$ and in conclusion
\[ \lim_{e\to \infty}\frac{b_e}{\rk(F^e_*R)}=\lim_{e\to \infty}\frac{\frk(F^e_*\omega_R)}{\rk(F^e_*R)}= s(\omega_R)=s(R)\rk(\omega_R)=s(R).
\]
The equality of $s(\omega_R)$ and $s(R)\rk(\omega_R)$ is the content of \cite[Theorem~4.11]{TuckerFSigExists}
\end{proof}

\begin{theorem}
\label{thm.CanonicalOrAntiCanonicalBlowup}
Suppose that an $F$-finite local ring $(R, \fram)$ is not Gorenstein and that either
 \begin{itemize}
 \item[(a)] $S= \bigoplus_nR(nK_R)$ is generated as a graded ring in degree $1$ or that
 \item[(b)]  $S = \bigoplus_n R(-nK_R)$ is generated as a graded ring in degree $1$.
 \end{itemize}
 Set $Y = \Proj S$ with $\pi : Y \to \Spec R$ the induced map.  Then we have $s(Y) \geq 2s(R)$.
\end{theorem}
\begin{proof}
The statement is trivial if $R$ is not strongly $F$-regular since then $s(R) = 0$.  Hence we may assume that $R$ is strongly $F$-regular.  In the case that $S = \bigoplus_n R(nK_R)$, we have that the small morphism $\pi : Y \to \Spec R$ is the blowup of $R(K_R)$ and hence $K_Y$ is Cartier.  In the case that $S = \bigoplus_n R(-nK_R)$ we have that $Y$ is the blowup of $R(-K_R)$ and so $-K_Y$ is Cartier, but then the inverse $K_Y$ is Cartier too.

Consider the split surjection
\[
F^e_* R \to R^{\oplus a_e} \oplus \omega_R^{\oplus b_e}
\]
guaranteed by \autoref{lem.HowManyOmegasSplit}.  We pull back via $\pi^*$ and reflexify and we obtain a split surjection
\[
F^e_* \O_Y \to \O_Y^{\oplus a_e} \oplus \O_Y(K_Y)^{\oplus b_e}.
\]
However, $\O_Y(K_Y)$ is locally free and hence the result follows since $b_e$ grows at the same rate as $a_e$ again by \autoref{lem.HowManyOmegasSplit}.
\end{proof}

\section{Behavior under more general blowups}
\label{sec.BehaviorMoreGeneralBlowups}

We now come to the proof of the more general case.
Throughout this section we work with varieties over an algebraically closed field ${\boldsymbol{k}}$ of positive characteristic $p$.  We begin with several lemmas.

\begin{lemma}
\label{lem:nohighercohomolvanishingalongpowers}
\textnormal{(\textit{cf.} \cite[Lemma 3.9]{LazarsfeldMustataConvexBodiesLinearSeries})}
Let $X$ be a projective variety, $x \in X$ a closed point of dimension $n$, and $A$ an ample Cartier divisor on $X$.  For all $1 \gg \epsilon > 0$ there exists $ \delta > 0$ such that
\begin{equation*}
h^i(X, \OO_X(kA) \otimes \m_x^{\lceil \epsilon k \rceil}) = 0 \mbox{ for } i > 0\mbox{, and}
\end{equation*}
\begin{equation*}
 h^0(X, \OO_X(kA)) - h^0(X, \OO_X(kA) \otimes \m_x^{\lceil \epsilon k \rceil})  \geq \delta \cdot k^n
\end{equation*}
for all $k \gg 1$.
\end{lemma}

\begin{proof}
Let $\mu \colon X' \to X$ be the blowup of $X$ along $\m_x$, with $\m_x \cdot \OO_{X'} = \OO_{X'}(-E)$.  Since $-E$ is $\mu$-ample, for a sufficiently large integer $m  > 1$ we have that $m \mu^* A - E$ is ample on $X'$. Shrinking $\epsilon$ if necessary, we may assume $m \epsilon  < 1$ and thus $\lceil \epsilon k \rceil m \leq k$ for $k \gg 0$. Since
\begin{equation*}
k\mu^* A - \lceil \epsilon k\rceil E = (k-\lceil \epsilon k\rceil m)\mu^* A - \lceil \epsilon k\rceil(m\mu^* A - E),
\end{equation*}
by Fujita vanishing \cite[Chapter 1.4.D, Theorem 1.4.35 and Remark 1.4.36]{LazarsfeldPositivity1} (using that $\mu^* A$ is nef) we have that
\begin{equation*}
 H^i(X', \OO_{X'}(k\mu^* A - \lceil \epsilon k\rceil E) = 0 \mbox{ for } i>0.
\end{equation*}
Recalling that
\begin{equation*}
\mu_*(\OO_{X'}(-\lceil \epsilon k \rceil E)) = \m_x^{\lceil \epsilon k \rceil } \; , \; R^j\mu_*(\OO_{X'}(-\lceil \epsilon k \rceil E)) = 0 \mbox{ for } j>0
\end{equation*}
provided $k \gg 1$ as shown \cite[Lemma 5.4.24]{LazarsfeldPositivity1}.  It follows that
\begin{equation*}
H^i(X, \OO_X(kA) \otimes \m_x^{\lceil \epsilon k \rceil }) = H^i(X', \OO_{X'}(k \mu^* A - \lceil \epsilon k\rceil E)) = 0 \mbox{ for } i > 0
\end{equation*}
when $k \gg 1$ by the vanishing above. In particular this holds for $i = 1$, and hence from the short exact sequence
\[
0 \to \OO_X(kA) \otimes \m_x^{\lceil \epsilon k \rceil} \to \OO_X(kA) \to \OO_{X,x} / \m_x^{\lceil \epsilon k \rceil} \to 0
\]
we have
\begin{equation*}
 h^0(X, \OO_X(kA)) - h^0(X, \OO_X(kA) \otimes \m_x^{\lceil \epsilon k \rceil})  = \dim_k(\OO_{X,x} / \m_x^{\lceil \epsilon k \rceil}) = P_{X,x}(\lceil \epsilon k \rceil)
\end{equation*}
where $P_{X,x}$ is the Hilbert-Samuel polynomial of $\OO_{X,x}$.  Thus, choosing $0 < \delta < \frac{\epsilon^n}{n!} e(\OO_{X,x})$ gives
\begin{equation*}
h^0(X, \OO_X(kA)) - h^0(X, \OO_X(kA) \otimes \m_x^{\lceil \epsilon k \rceil})  = P_{X,x}(\lceil \epsilon k \rceil) > \frac{\delta}{\epsilon^n} (\lceil \epsilon k \rceil)^n \geq \delta k^n
\end{equation*}
for $k \gg 1$.
\end{proof}

\begin{lemma}
\label{lem:makeH1growbasic}
\textnormal{(\textit{cf.} \cite[Lemma 2.9]{LiuXuKStabilityOfCubicThreefolds})}
Let $Y$ be a normal projective variety of dimension $n$ over a field ${\boldsymbol{k}}$ of prime characteristic $p>0$, and $L$ a nef and big Cartier divisor on $Y$.  Let $y \in Y$ be a closed point of an irreducible curve $C$ satisfying $(L \cdot C) = 0$.  Then there exists $\epsilon > 0$ so that
\begin{equation*}
    h^1(Y,\OO_Y(kL) \otimes \m_y^{k}) \geq \epsilon k^n \mbox{ for } k \gg 1.
\end{equation*}
\end{lemma}

The following proof was provided to us by Takumi Murayama.  We will provide an alternative (and somewhat longer) proof below.

\begin{proof}
Let $\psi: \hat{Y}\to Y$ be the normalized blowup of $Y$ along $\m_y$ and let $\m_y\O_{\hat{Y}}=\O_{\hat{Y}}(-E)$. Let $\hat{C}$ be the strict transform of $C$ in $\hat{Y}$, in which case $\hat{C}\cdot E >0$. Let $A$ be a very ample Cartier divisor on $\hat{Y}$ so that the $\Q$-Cartier divisor $\psi^*L-E+\delta A$ is not ample for $1\gg \delta >0$ since
\[
(\psi^*L-E+\delta A)\cdot \hat{C}=-E\cdot \hat{C}+\delta A\cdot \hat{C},
\]
which is negative for all $1\gg \delta>0$.

Fix $1\gg \delta >0$. Since $\psi^* L -E +\delta A$ is not ample and $\psi^* L -E=\psi^* L -E+\delta A-\delta A$ there exists some $i>0$ and $\epsilon>0$ such that
\[
h^i(\hat{Y}, \O_{\hat{Y}}(m(\psi^*L-E)))\geq \epsilon m^n
\]
for all $m\gg 0$ by \cite[Theorem~B]{MurayamaGammaConstruction}. By \cite[Lemma~5.4.24]{LazarsfeldPositivity1}, for $k\gg0$ we have
\[
h^i(\hat{Y},\O_{\hat{Y}}(k(\psi^*L-E))=h^i(Y, \O_Y(kL)\otimes \m_y^k)).
\]
But if $i\geq 2$ then the exact sequence of cohomology derived from twisting the short exact
\[
0\to \m_y^k\to \O_Y\to \O_Y/\m_y^k\to 0
\]
by $kL$ shows
\[
h^i(Y, \O_Y(kL)\otimes \m_y^k))= h^i(Y, \O_Y(kL))
\]
for all $i\geq 2$. By \cite[Theorem~1.4.40]{LazarsfeldPositivity1}
\[
h^i(Y, \O_Y(kL))=O(k^{n-i}).
\]
Therefore it is only possible for $h^i(Y, \O_Y(kL)\otimes \m_y^k))\geq \epsilon k^n$ for all $m\gg0$ when $i=1$ which completes the proof of the lemma.
\end{proof}

The above proof of Lemma~\ref{lem:makeH1growbasic} provides an alternative proof to \cite[Lemma~2.9]{LiuXuKStabilityOfCubicThreefolds}. One would need to replace the reference of \cite[Theorem~B]{MurayamaGammaConstruction} with \cite[Theorem~A]{DeFernexKuronyaLazarsfeldHigherCohomology}. Nevertheless, we present a second proof of Lemma~\ref{lem:makeH1growbasic} which closely resembles the proof of \cite[Lemma~2.9]{LiuXuKStabilityOfCubicThreefolds}.  We suspect the alternative proof will be of independent interest. 

\begin{lemma}
\label{lem:nonneflocus}
Suppose $V$ is a normal projective variety, and $D$ is a $\QQ$-Cartier $\QQ$-divisor with non-negative Iitaka dimension that is not nef.  Let $Z \subseteq V$ be an irreducible curve such that $D \cdot Z < 0$.  Let $g \colon W' \to V$ be a regular alteration dominating the blowup of $I_Z$ such that $g^{-1}(Z)$ has simple normal crossings.  Then $\tau( W', m \parallel g^* D\parallel)$ vanishes along $g^{-1}(Z)$ for all integers $m \gg 1$.  In particular, if $D$ is big, every irreducible component of $g^{-1}(Z)$ is contained in the non-nef locus of $g^*(D)$.
\end{lemma}

\begin{proof}
Replacing $D$ with a positive multiple, we may assume that $D$ is a Cartier divisor.
Let $\mu \colon V' \to V$ be the normalized blowup of $I_Z$, with $I_Z \OO_{V'} = \OO_{V'}(-E)$ and $f' \colon W' \to V'$ the induced map factoring $g$, so that the divisor $E' = (f')^*E$ has simple normal crossing support.  Let $g \colon W' \xrightarrow{\nu} W \xrightarrow{f}V$ be the Stein factorization of $g$ (in other words, $W := {\bf Spec}\; g_* \O_{W'}$), so that we have a commutative diagram
\begin{equation*}
    \xymatrix{
    W' \ar[r]^{f'} \ar[d]_{\nu} \ar[dr]|-{g} & V'\ar[d]^{\mu} \\
    W \ar[r]_{f} & V
    }
\end{equation*}
where $f$ is finite, $\nu$ is birational, and $W$ is normal.  Since $f$ is finite, $f^{-1}(\{Z\})$ is a union of finitely many irreducible curves $Z_1, \dots, Z_r$ that dominate $Z$.  Note that, if  $C \subseteq W'$ is an irreducible curve that dominates $Z$, we have by the projection formula that $(g^*D)\cdot C = (\deg g|_C)(D \cdot Z) < 0$.

Given $m \geq 1$, consider the asymptotic test ideal $\tau(W', m \parallel g^* D\parallel ) \subseteq \OO_{W'}$.  If $H$ is a very ample divisor on $W'$ and $A = K_{W'} + (\dim W' + 1)H$, then
\begin{equation*}
    \OO_{W'}(m (g^*D) + A) \otimes \tau(m\parallel g^*D \parallel)
\end{equation*}
is globally generated for all $m \geq 1$ by \cite[Theorem A]{MustataNonNefLocusPositiveChar}.
Therefore, if $C \subseteq W'$ is an irreducible curve that is not contained in the zero locus of $\tau(W', m \parallel g^*D \parallel)$, then
\begin{equation}
\label{eq.mCondition}
    (m(g^*D) + A) \cdot C \geq 0.
\end{equation}
Thus, if $C$ dominates $Z$, and hence $(g^*D) \cdot C < 0$ similarly to the above, we must have that $C$ is contained in the zero locus of $\tau(W', m \parallel g^*D \parallel)$ for all $m > -(A \cdot C)/ ((g^*D) \cdot C)$.  Note this condition on $m$ comes from negating \autoref{eq.mCondition} and solving for $m$.

Consider a component $E_i'$ of $E'$ that dominates $Z$.  A general complete intersection curve $C$ on $E_i'$ then dominates $Z$, and thus $\tau(W', m_i \parallel g^*D \parallel)$ must vanish along $C$ for some $m_i \gg 1$. As we vary the complete intersection that defines $C$, the condition on $m_i$ does not change.  Thus, in fact, $\tau(W', m_i \parallel g^*D \parallel)$ must vanish along all of $E_i'$.

Supposing now that $E_i'$ is a component of $E'$ that maps to a point of $Z$, we again wish to show that $\tau(W', m_i \parallel g^*D \parallel)$ must vanish along all of $E_i'$ for $m_i \gg 1$. We have that $E_i'$ necessarily also maps to a point of $Z_s \subseteq f^{-1}(Z) \subseteq W$ for some $s$.  Note that some component of $E'$ must dominate $Z_s$ (since $\nu$ is surjective and $\nu^{-1}(Z_s) \subseteq \Supp(E')$) and $\nu^{-1}(Z_s)$ is connected as $W$ is normal.
In light of the previous paragraph, it suffices to show $\tau(W', m_i \parallel g^*D \parallel)$ must vanish along all of $E_i'$ for $m_i \gg 1$.  
We may assume that $E_i'$ intersects another component $E_j'$ of $E'$ along which $\tau(W', m_j \parallel g^*D \parallel)$ is known to vanish for some $m_j \gg 0$.

Take a general complete intersection curve $C \subseteq E_i'$ that meets $E_j'$ in at least one point $P$, which we may assume to be a smooth point of $C$.  We know that
\begin{equation*}
    \OO_{W'}(lm_j (g^*D) + A) \otimes \tau(W', lm_j \parallel g^*D \parallel)
\end{equation*}
is globally generated for any $l \geq 1$.  Thus, whenever $\tau(W', lm_j \parallel g^*D \parallel)$ does not vanish along $C$, we can find an effective divisor $F \sim_{\ZZ} (lm_j (g^*D) + A)$ not containing $C$ that vanishes along $\tau(W', lm_j \parallel g^*D \parallel)$.  Let us consider what happens when we restrict $F$ to $C$.  Note that since $E_i'$ maps to a point of $Z$, so too does $C \subseteq E_i'$, whence $(g^*D)\cdot C  = 0$.
Furthermore,  $\tau(W',  m_j \parallel g^*D \parallel) \subseteq \OO_{W'}(-E_j')$ by assumption, so we have that
\[
\tau(W',  lm_j \parallel g^*D \parallel) \subseteq \tau(W', m_j \parallel g^*D \parallel)^l \subseteq \OO_{W'}(-lE_j')
\]
for all $l \geq 1$ by subadditivity \cite[Theorem 4.5]{HaraYoshidaGeneralizationOfTightClosure}.
Thus, $F$ must vanish at least to order $l$ at $P$, so that
\[
A\cdot C = (0 + A)\cdot C= F \cdot C \geq l.
\]
But $A$ does not depend on $l$, so this is impossible, and so $\tau(W', lm_j \parallel g^*D \parallel)$ vanishes along $C$.
Fix $l > A\cdot C$ and set $m_i = lm_j$.  It follows $\tau(W', m_i \parallel g^*D \parallel)$ must vanish along $C$ and hence also $E_i'$, as desired.

Thus, taking $m'$ sufficiently large and divisible, we conclude from above that
\begin{equation*}
    \tau(W', m' \parallel g^*D \parallel) \subseteq \OO_{W'}(-E'_{\mathrm{red}})
\end{equation*}
so that $\tau(W', m \parallel g^* D\parallel)$ vanishes along $g^{-1}(Z) = E'_{\mathrm{red}}$ for all integers $m \gg 1$.  In particular, if $D$ is big, every irreducible component $E_i'$ of $g^{-1}(Z)$ is contained in the non-nef locus of $g_*(D)$ by  \cite[Theorem 6.2]{MustataNonNefLocusPositiveChar}.
\end{proof}

\begin{lemma}
\label{lem:baselocusofcurve}
\textnormal{(\textit{cf.} \cite[Proposition 1.1]{DeFernexKuronyaLazarsfeldHigherCohomology} and \cite[Proposition 4.5]{MurayamaGammaConstruction})}
Suppose that $V$ is a normal projective variety and $Z \subseteq V$ is an irreducible curve.  Let $L$ and $E$ be Cartier divisors, with $L$ big and $E$ effective. Assume $L\cdot Z = 0$, that $E$ does not contain $Z$ and that {{$E \cdot Z > 0$}}.

If $0 < \gamma_1 < \gamma_2$ are real numbers  such that $L - \gamma_2 E$ remains big, then there exists $\epsilon > 0$ and a positive integer $c$ such that $\b(| kL - mE|) \subseteq I_Z^{\lfloor \epsilon k \rfloor - c}$ for all integers $m$ and $k$ such that $\gamma_1 k \leq m \leq \gamma_2 k$.
\end{lemma}

\begin{proof}  Without loss of generality, we assume that the base field ${\boldsymbol{k}} = \overline{{\boldsymbol{k}}}$ is uncountable.
Using \cite[Theorem 4.1]{deJongAlterations}, we may take a regular alteration
$g \colon W' \to V$  dominating the blowup of $I_Z$ such that $g^{-1}(Z)$ has simple normal crossings. Let $\mu \colon V' \to V$ be the normalized blowup of $I_Z$, with $I_Z \OO_{V'} = \OO_{V'}(-G)$ and  $f' \colon W' \to V'$ the induced map factoring $g$ so that $g = \mu \circ f'$.
For any $t \in \QQ \cap [\gamma_1, \gamma_2]$, $L - tE$ is big and $(L-tE)\cdot Z = -t(E \cdot Z) < 0$. Applying \autoref{lem:nonneflocus}, it follows that every irreducible component $g^{-1}(Z)$ is contained in the non-nef locus of $g^*(L-tE)$.

Thus, if $(f')^*G = G' = \sum a_i G_i'$ so that $(g^{-1}(Z))_{\reduced} = G'_{\mathrm{red}}$, it follows from \cite[Theorem 6.2]{MustataNonNefLocusPositiveChar} that
\begin{equation*}
 \ord_{G_i'}(\parallel g^*L - t g^*E\parallel ) = \inf_{l \geq 1 \atop tl \in \ZZ} \frac{1}{l} \ord_{G_i'} \big(\b(| l (g^*L - t g^*E)|)\big) > 0
\end{equation*}
for all $t \in \QQ \cap [\gamma_1, \gamma_2]$ and any $i$.  Since the asymptotic order of vanishing $\ord_{G_i'}(\parallel - \parallel)$ is continuous on the open cone of big divisors in $N^1(X)_{\RR}$ by \cite[Theorem 6.1]{MustataNonNefLocusPositiveChar}, there exists an $\epsilon' > 0$ so that
\begin{equation*}
 \ord_{G_i'}(\parallel g^*L - t g^*E\parallel ) > \epsilon'
\end{equation*}
for all $t \in [\gamma_1, \gamma_2]$ and any $i$.  In particular, we have that
\begin{equation*}
 \ord_{G'_i}\big(\b(|kg^*L - mg^*E|)\big) \geq k \epsilon'
\end{equation*}
for all integers $m,k$ satisfying $\gamma_1 k \leq m \leq \gamma_2 k$.  In this case, setting $a = \max_i a_i$ (the largest coefficient of $(f')^* G$) and $\epsilon = \epsilon'/a$ gives
\begin{equation*}
 \b(|kg^*L - mg^*E|) \subseteq \OO_{W'}(- \lfloor \epsilon k\rfloor G').
\end{equation*}
Since $(f')^*|k\mu^*L - m \mu^*E| \subseteq |(kg^*L - mg^*E)|$, we have
\begin{equation*}
\b(|k\mu^*L - m \mu^*E|)\cdot \OO_{W'} \subseteq \b(|kg^*L - mg^*E)|)  \subseteq \OO_{W'}(- \lfloor \epsilon k\rfloor G')
\end{equation*}
and pushing forward along $f'$ gives
\begin{equation*}
 \b(|k\mu^*L - m \mu^*E|)\cdot (f')_*\OO_{W'} \subseteq  (f')_*\OO_{W'}(-\lfloor \epsilon k\rfloor G') = \OO_{V'}(-\lfloor \epsilon k\rfloor G)\cdot (f')_*\OO_{W'}.
\end{equation*}
Thus, using that $\OO_{V'}$ is normal and  $\OO_{V'} \subseteq (f')_*\OO_{W'}$ is a finite and hence integral extension, we see
\begin{equation*}
\b(|k\mu^*L - m \mu^*E|) \subseteq \left( \OO_{V'}(-\lfloor \epsilon k\rfloor G)\cdot (f')_*\OO_{W'} \right) \cap \OO_{V'} = \OO_{V'}(-\lfloor  \epsilon k\rfloor G)
\end{equation*}
from \cite[Propositions 1.5.2 and 1.6.1]{HunekeSwansonIntegralClosure}.  On the other hand, we have $H^0(V, \OO_V(kL - mE)) = H^0(V', \OO_{V'}(k\mu^*L - m \mu^*E))$ again by normality, and in particular
\begin{equation*}
 \b(|kL - m E|) \cdot \OO_{V'} = \b(|k\mu^*L - m \mu^*E|).
\end{equation*}
Pushing forward along $\mu : V' \to V$ then gives $\b(|kL - m E|) \subseteq \overline{I_Z^{\lfloor \epsilon k \rfloor}}$.  Using \cite[Proposition 5.3.4]{HunekeSwansonIntegralClosure} there exists a positive integer $c$ so that $\overline{I_Z^\ell} \subseteq I_Z^{\ell-c}$ for all integers $\ell \geq c$, and the result now follows.
\end{proof}

\begin{proof}[Second proof of Lemma~\ref{lem:makeH1growbasic}]
We may assume that ${\boldsymbol{k}} = \overline{{\boldsymbol{k}}}$ is an uncountable field of prime characteristic. Let $\psi \colon \hat{Y} \to Y$ be the normalized blowup of $Y$ along $\m_y$, with $\m_y\cdot \OO_{\hat{Y}} = \OO_{\hat{Y}}(-E)$. Take $\hat{C}$ to be the strict transform of $C$ in $\hat{Y}$, noting that $\hat{C}$ is not contained in $E$ and {{$\hat{C} \cdot E > 0$}}.   We have that $\psi^*L$ is big with $\psi^*L \cdot \hat{C} = L \cdot C = 0$.  Moreover for some sufficiently large integer $\ell > 0$ we have that $\ell \psi^* L - E$ is also big.  Set $\gamma_2 = 1/\ell$ and choose $0 < \gamma_1 < \gamma_2$.

Consider the long exact sequence
\[
\dots \to H^1(Y, \OO_Y(kL) \otimes \m_y^{k}) \to H^1(Y, \OO_Y(kL) \otimes \overline{\m_y^{k}}) \to H^1(Y, \OO_Y(kL) \otimes \overline{\m_y^{k}}/\m_y^k) = 0
\]
where the vanishing holds since $(\overline{\m_y^k}) / \m_y^k$ is a skyscraper sheaf with support contained in $\{ y \}$.  Using this sequence and the fact that $R^j\psi_*\OO_{\hat{Y}}(-kE) = 0$ for $j > 0$ and sufficiently large $k$ (as $-E$ is $\psi$-ample), we have
\begin{equation}
\label{eq.ComparisonVsBlowupH1Sections}
  h^1(Y,\OO_Y(kL) \otimes \m_y^{k}) \geq h^1(Y,\psi_*\OO_{\hat{Y}}(k\psi^*L - kE)) = h^1(\hat{Y},\OO_{\hat{Y}}(k\psi^*L - kE))
\end{equation}
for all $k \gg 1$.  Consider now the differences
\begin{equation*}
 \Delta_k(m) := h^1(\hat{Y},\OO_{\hat{Y}}(k\psi^*L - (m+1)E)) -  h^1(\hat{Y},\OO_{\hat{Y}}(k\psi^*L - mE))
\end{equation*}
for $k \geq m > 0$.  If $m \gg 1$, and using that $\OO_{\hat{Y}}(\psi^*L)|_E = \OO_E$ and $\OO_{\hat{Y}}(-E)|_{E} \sim \OO_{E}(1)$ is ample, we have that
\begin{equation*}
h^1\left(E,\OO_{\hat{Y}}\left((k\psi^*L - mE)|_{E}\right)\right) = h^1(E, \OO_{E}(m)) = 0
\end{equation*}
using Serre vanishing.  Thus, if $m \geq \gamma_1k$ and $k \gg 1$, it follows that $\Delta_k(m) \geq 0$.  On the other hand, if additionally $\gamma_2 k  \geq m \geq \gamma_1 k$, we have from \autoref{lem:baselocusofcurve} that there is some $\epsilon' > 0$ and a positive integer $c$ so that
\begin{equation*}
 \b(|k\psi^*L - mE|) \subseteq I_{\hat{C}}^{\lfloor \epsilon'k\rfloor - c}.
\end{equation*}
Thus, if $\hat{y} \in \hat{C} \cap \Supp(E)$ is a closed point, we have an inclusion
\begin{equation}
\label{eq.ImageIsContainedInIdealPower}
\begin{array}{rl}
   & \IM\left(H^0(\hat{Y}, \OO_{\hat{Y}}(k\psi^* L - mE)) \to H^0(E,\OO_{\hat{Y}}(k\psi^*L-mE)|_E)\right)\\
   \subseteq & H^0(E,\OO_{E}(m) \otimes \m_{\hat{y}}^{\lfloor \epsilon'k\rfloor - c}).
   \end{array}
\end{equation}
Choosing $0 < \epsilon'' < \epsilon'/\gamma_2$, we have that for $k \gg 1$,
\begin{equation}
\label{eq.AnnoyingInequalities}
 \lfloor \epsilon' k \rfloor - c > \epsilon'k - c - 1 \geq \epsilon'' \gamma_2k + 1 \geq \epsilon'' m + 1 > \lceil \epsilon'' m \rceil.
\end{equation}

Shrinking $\epsilon''$ further if necessary, by \autoref{lem:nohighercohomolvanishingalongpowers} there exists $\delta > 0$ such that
\begin{eqnarray*}
 \Delta_k(m) &=& h^0(E,\OO_E(m)) - \rk\left(H^0(\hat{Y}, \OO_{\hat{Y}}(k\psi^* L - mE)) \to H^0(E,\OO_{\hat{Y}}(k\psi^*L-mE)|_E)\right) \\
 & \geq & h^0(E, \OO_E(m))-h^0(E, \OO_E(m)\otimes \m_{\hat{y}}^{\lfloor \epsilon'k\rfloor - c})\;\; \text{\hfill by \autoref{eq.ImageIsContainedInIdealPower}} \\
 & \geq & h^0(E, \OO_E(m))-h^0(E, \OO_E(m)\otimes \m_{\hat{y}}^{\lceil \epsilon'' m \rceil})\;\;\text{by \autoref{eq.AnnoyingInequalities}} \\
 & \geq & \delta m^{n-1}
\end{eqnarray*}
for all $\gamma_1 k \leq m \leq \gamma_2 k$ and $k \gg 1$.  Thus, we compute
\begin{eqnarray*}
 h^1(\hat{Y},\OO_{\hat{Y}}(k\psi^*L - kE)) & = & \left( \sum_{m=\lceil \gamma_1 k \rceil}^{k-1}\Delta_k(m) \right) + h^1(\hat{Y},\OO_{\hat{Y}}(k\psi^*L - \lceil \gamma_1 k \rceil E)) \\
 & \geq & \sum_{m=\lceil \gamma_1 k \rceil}^{\lceil \gamma_2 k \rceil - 1} \Delta_k(m) \text{ (since the dropped $\Delta_k(m) \geq 0$)} \\
 & \geq & \sum_{m=\lceil \gamma_1 k \rceil}^{\lceil \gamma_2 k \rceil - 1} \delta m^{n-1} \\
 & \geq & \delta (\lceil \gamma_1 k \rceil)^{n-1}(\lceil \gamma_2k \rceil -\lceil \gamma_1k \rceil) \\
 & \geq & \delta (\lceil \gamma_1 k \rceil)^{n-1}(\gamma_2k - 1  -\gamma_1k ) \\
 & \geq & \delta \gamma_1^{n-1}(\gamma_2 - \gamma_1)k^n - \delta \gamma_1^{n-1}k^{n-1}
\end{eqnarray*}
for all $k \gg1$. Thus, choosing $\epsilon < \delta \gamma_1^{n-1}(\gamma_2 - \gamma_1)$ implies that
\[
\epsilon k^n < \delta \gamma_1^{n-1}(\gamma_2 - \gamma_1)k^n - \delta \gamma_1^{n-1}k^{n-1} \leq h^1(\hat{Y},\OO_{\hat{Y}}(k\psi^*L - kE))
\]
for $k \gg 1$.  Therefore by \autoref{eq.ComparisonVsBlowupH1Sections},
\begin{equation*}
  h^1(Y,\OO_Y(kL) \otimes \m_y^{k}) \geq  h^1(\hat{Y},\OO_{\hat{Y}}(k\psi^*L - kE)) \geq \epsilon k^n
\end{equation*}
for all $k \gg 1$ as desired.
\end{proof}

Now we come to the main theorem of the section.

\begin{theorem}
\label{thm:secondgoal}
Let $X$ be a strongly $F$-regular variety of dimension $n$ over an algebraically closed field ${\boldsymbol{k}}$ of characteristic $p>0$.  Suppose $\pi \colon Y \to X$ is a proper birational morphism from a normal variety $Y$ and fix a point $y \in \mathrm{Exc}(\pi)$ with $\pi(y)=x$.  Suppose additionally that either:
\begin{enumerate}
    \item $\pi$ is small, \textit{i.e.} $\pi$ is an isomorphism outside of a set of codimension at least two in $Y$, or;
    \item The canonical divisor $K_X$ is $\QQ$-Cartier and for every exceptional divisor $E$ containing $y$, we have that \mbox{$\mathrm{coeff}_E(K_Y-\pi^*K_X) \leq 0$}.  For instance, this holds if all the discrepancies are non-positive.
\end{enumerate}
Then we have $s(\OO_{X,x}) < s(\OO_{Y,y})$.
\end{theorem}

\begin{proof}
If ${\boldsymbol{k}}$ is not uncountable then we base change by the field obtained by adjoining uncountably many indeterminants to ${\boldsymbol{k}}$ and then taking its algebraic closure.  Any closed points on the original varieties will correspond to points on the base-changed varieties, and their signatures will not change by \cite[Theorem~5.4]{YaoObservationsAboutTheFSignature}.  Thus we may assume that ${\boldsymbol{k}}$ is uncountable and algebraically closed.

Set $R = \OO_{X,x}$ and $S = \OO_{Y,y}$ so that we have a local inclusion $R \subseteq S$.  By the assumption that $\pi$ is either small or has non-positive discrepancy at $y$, it follows that $p^{-e}$-linear map on $R$ extends naturally to a $p^{-e}$-linear map on $S$, see \autoref{lem.ExtendingMapsBirationally}.  Consider the Frobenius degeneracy ideals $I_e^S$ of $S$ used to define the $F$-signature \autoref{sec.Preliminaries}, so that $\displaystyle{s(\OO_{Y,y}) = \lim_{e \to \infty} \frac{1}{p^{ne}}\ell(S/I_e^S)}$, and similarly for the Frobenius degeneracy ideals $I_e^R$ of $R$. Set $J_e = I_e^S \cap R$. Observe that, if $\m_R$ can be generated by $d$ elements, we have
\begin{equation*}
    \m_R^{dp^e} \subseteq \m_R^{[p^e]} \subseteq J_e \subseteq I_e^R.
\end{equation*}
Indeed, the first inclusion is standard by looking one monomial in the generators at a time.  The second inclusion follows from the fact that $\m_R^{[p^e]}\subseteq \m_S^{[p^e]} \subseteq I_e^S$.  For the last inclusion, suppose that $r \in R \setminus I_e^R$.  Then we know there exists a $p^{-e}$-linear map $\phi$ on $R$ so that $\phi(r) = 1$.  But then $\phi$ extends to $S$, and we still have $\phi(r) = 1$, so that $r \not\in I_e^S$.  Note also that $J_e^{[p]} \subseteq J_{e+1}$, so that $\displaystyle{\lim_{e \to \infty} \frac{1}{p^{ne}} \ell(R/J_e)}$ exists and is at least as large as $s(R) = s(\OO_{X,x})$, see \cite[Theorem~B]{PolstraTuckerCombined}.

Let us take suitable projective closures of $X,Y$ such that $\pi$ extends to a birational morphism between normal projective varieties.  Note that conditions (a) or (b) from the statement of the theorem will not necessarily hold on the entire compactifications, however we will not need this.  Let $M'$ be an ample line bundle on $X$.  By \autoref{lem:nohighercohomolvanishingalongpowers}, for all $1 \gg \epsilon' > 0$, $i > 0$ and $k \gg 1$ we have $H^i(X, (M')^{\otimes k}\otimes \m_R^{\lceil \epsilon' k \rceil})=0$.  Taking $\ell \gg 1$ so that $1/\ell < \epsilon'$, it follows that $H^i(X, (M')^{\otimes \ell dp^e}\otimes \m_R^{dp^e})=0$ for $i > 0$ and $e \gg 1$.  Setting $M = M'^{\otimes \ell d}$, and using that $J_e / \m_R^{dp^e}$ is supported only at $x \in X$, it follows that
\begin{equation*}
    \begin{array}{l}
    H^1(X, M^{\otimes p^e}\otimes \m_R^{dp^e}) \twoheadrightarrow H^1(X, M^{\otimes p^e} \otimes J_e),
    \\
    H^i(X, M^{\otimes p^e}\otimes \m_R^{dp^e}) \overset{\cong}{\to} H^i(X, M^{\otimes p^e} \otimes J_e) \mbox{ for } i \geq 2.
    \end{array}
\end{equation*}
and hence $H^i(X,M^{\otimes p^e}\otimes J_e) = 0$ for $i > 0$ and $e \gg 1$.  Thus we have
\begin{equation}
\label{eq:downstairs}
    \lim_{e \to \infty} \frac{1}{p^{en}}h^0(X,M^{\otimes p^e}\otimes J_e) = \frac{1}{n!}\vol_X(M) - \lim_{e \to \infty} \frac{1}{p^{ne}} \ell(R/J_e).
\end{equation}
On the other hand, since $X$ is strongly $F$-regular, so too is $Y$ and it follows from the proof of the positivity of the $F$-signature that there is some $e_0$ with $I_e^S \subseteq \m_S^{p^{e - e_0}}$ for all $e \gg 1$ (\cite[Theorem~3.21]{BlickleSchwedeTuckerFSigPairs1}, \textit{cf.} \cite[Section 5 and the Second proof of Theorem 5.1]{PolstraTuckerCombined}).  We have the following relations
\begin{equation*}
    \begin{array}{l}
    H^1(Y, \pi^*M^{\otimes p^e}\otimes I_e^S) \twoheadrightarrow H^1(Y, \pi^*M^{\otimes p^e} \otimes \m_S^{p^{e-e_0}}),
    \\
    H^i(Y, \pi^*M^{\otimes p^e}\otimes I_e^S) \overset{\cong}{\to} H^i(Y, \pi^*M^{\otimes p^e}) \mbox{ for } i \geq 2.
    \end{array}
\end{equation*}
Since $h^i(Y, \pi^*M^{\otimes p^e}) = O(p^{e(n-1)})$ for $i > 0$ as $\pi^*M$ is nef \cite[Theorem~1.4.40]{LazarsfeldPositivity1}, we have that
\begin{equation*}
\begin{array}{l}
\displaystyle
    \limsup_{e\to \infty} \frac{1}{p^{en}} h^0(Y, \pi^*M^{\otimes p^e}\otimes I_e^S) \\ \displaystyle = \frac{1}{n!}\vol_Y(\pi^*M) - s(\OO_{Y,y})+ \limsup_{e \to \infty} \frac{1}{p^{en}} h^1(Y, \pi^*M^{\otimes p^e}\otimes I_e^S)\\ \displaystyle \geq \frac{1}{n!}\vol_X(M) -  s(\OO_{Y,y}) + \limsup_{e \to \infty} \frac{1}{p^{en}} h^1(Y, \pi^*M^{\otimes p^e}\otimes \m_S^{p^{e-e_0}}).
\end{array}
\end{equation*}
By \autoref{lem:makeH1growbasic} applied with $L = \pi^*M^{\otimes p^{e_0}}$, there exists an $\epsilon > 0$ so that
\[
h^1(Y, \pi^*M^{\otimes p^e}\otimes \m_S^{p^{e-e_0}}) = h^1(Y, L^{\otimes(p^{e-e_0})}\otimes \m_S^{p^{e-e_0}}) \geq \epsilon p^{(e - e_0)n}
\]
for all $e \gg 1,$
so that
\begin{equation}
    \label{eq:upstairs}
    \limsup_{e\to \infty} \frac{1}{p^{en}} h^0(Y, \pi^*M^{\otimes p^e}\otimes I_e^S) \geq \frac{1}{n!}\vol_X(M)-  s(\OO_{Y,y}) +\frac{\epsilon}{p^{ne_0}}.
\end{equation}
Observe that $\pi_* I_e^S \subseteq \pi_* \O_Y = \O_X$ and so $J_e = \pi_* I_e^S$ which implies that
\[
\limsup_{e\to \infty} \frac{1}{p^{en}} h^0(Y, \pi^*M^{\otimes p^e}\otimes I_e^S) = \lim_{e \to \infty} \frac{1}{p^{en}}h^0(X,M^{\otimes p^e}\otimes J_e).
\]
Thus combining \eqref{eq:upstairs} and \eqref{eq:downstairs} we have
\[
\frac{1}{n!}\vol_X(M) - \lim_{e \to \infty} \frac{1}{p^{ne}} \ell(R/J_e) \geq \frac{1}{n!}\vol_X(M)-  s(\OO_{Y,y}) +\frac{\epsilon}{p^{ne_0}}
\]
whence it follows
\begin{equation*}
   s(\OO_{Y,y}) \geq \lim_{e \to \infty} \frac{1}{p^{ne}} \ell(R/J_e) + \frac{\epsilon}{p^{ne_0}} > s(\OO_{X,x}).
\end{equation*}
This completes the proof.
\end{proof}

\section{Examples of prime characteristic invariants and blow-ups of isolated singularities}
\label{sec.Examples}

In this section we observe that, without the hypothesis (a) or (b), the conclusion of Theorem~\ref{thm:secondgoal} may not hold even if $\pi: Y\to X$ is the blow-up of an isolated singularity. We provide several examples demonstrating various negative behaviors. We fix the following notation for all of our examples:  $X$ will be an affine scheme of a strongly $F$-regular hypersurface. Specifically, $X=\Spec(R)$, $R={\boldsymbol{k}}[x_1,\ldots,x_n]/(f)$, ${\boldsymbol{k}}$ will be an algebraically closed field of prime characteristic $p>0$, and $X$ will have isolated singularity at the origin $(x_1,\ldots,x_n)$. We denote by $\pi: Y\to X$ the blow-up of $X$ at the origin. Then $\pi$ is proper, birational, and has $n$ standard affine charts:
\[
Y_i=\Spec\left(R\left[\frac{x_1}{x_i},\ldots, \frac{x_n}{x_i}\right]\right)\cong \Spec\left(\frac{k\left[\frac{x_1}{x_i},\ldots, x_i,\ldots \frac{x_n}{x_i}\right]}{(f:x_i^\infty)}\right)
\]
where $(f:x_i^\infty)=\bigcup_{\ell\in \NN}\{g\in k\left[\frac{x_1}{x_i},\ldots, x_i,\ldots \frac{x_n}{x_i}\right]\mid x_i^\ell g\in (f)\}$.

Our strategy of showing $F$-signature can strictly decrease under the blow-up of an isolated singularity avoids any technical computations or explicit formulas of $F$-signature. Instead, we show a strongly $F$-regular isolated singularity can be blown-up to create a variety which has non-strongly $F$-regular points.  We first discuss a method of determining if an isolated hypersurface singularity is strongly $F$-regular.

\begin{lemma}\label{lem: criterion for strong F-regularity} Let ${\boldsymbol{k}}$ be an $F$-finite field of prime characteristic $p>0$, $S={\boldsymbol{k}}[x_1,\ldots,x_n]$, and $f\in S$ an element such that $S/(f)$ is a domain with isolated singularity at the maximal ideal $(x_1,\ldots,x_n)$. Then $S/(f)$ is strongly $F$-regular if and only if there exists $e\in \NN$ such that $x_1 f^{p^e-1}\not \in (x_1^{p^e},\ldots ,x_n^{p^e})$.
\end{lemma}

\begin{proof}
The property of being strongly $F$-regular is a local condition. Let $\m=(x_1,\ldots, x_n)$. Then $R$ is strongly $F$-regular if and only if $R_\m$ is a strongly $F$-regular local ring. By \cite{AberbachEnescuStructureOfFPure} the set
\[
\mathcal{P}=\bigcap_{e\in\NN}\{c\in R_\m\mid R^{1/p^e}_\m\xrightarrow{\cdot c^{1/p^e}}R^{1/p^e}_\m\mbox{ does not split}\}
\]
is an ideal of $R_\m$ satisfying the following:
\begin{enumerate}
\item $R_\m$ is $F$-pure if and only if $\mathcal{P}\not = R_\m$;
\item If $R_\m$ is $F$-pure  then $\mathcal{P}$ is a prime ideal;
\item $R_\m$ is not strongly $F$-regular then the closed set $V(\mathcal{P})$ of $\Spec(R_\m)$ defines the non-strongly F-regular locus of $R_\m$.
\end{enumerate}
Thus the assumption that $R$ has isolated singularity implies that $\mathcal{P}$ is $0$ if $R$ is strongly $F$-regular, the unique maximal ideal of $R_\m$ if $R$ is $F$-pure but not strongly $F$-regular, or all of $R_\m$ if $R$ is not $F$-pure. Therefore $R_\m$ is strongly $F$-regular if and only if $x_1\not\in \mathcal{P}$.  It readily follows by the techniques of \cite{FedderFPureRat} that $x_1\not\in \mathcal{P}$ if and only if there exists $e\in \NN$ such that $x_1f^{p^{e}-1}\not \in (x_1^{p^e},\ldots, x_n^{p^e})S_\m$ (c.f. \cite[Theorem~2.3]{GlassbrennerSFRinImagesOfRegular}). Since $(x_1^{p^e},\ldots, x_n^{p^e})S$ is primary to $\m$ we have $x_1f^{p^{e}-1}\not \in (x_1^{p^e},\ldots, x_n^{p^e})S_\m$ if and only if $x_1f^{p^{e}-1}\not \in (x_1^{p^e},\ldots, x_n^{p^e})S$.
\end{proof}

\begin{example}\label{example: first example} Let
\[
R=\frac{{\boldsymbol{k}}[x_1,x_2,x_3,x_4]}{(x_1^2+x_2^4+x_3^5+x_4^4)}
\]
and assume that ${\boldsymbol{k}}$ is an algebraically closed field of characteristic $7$. Then $z(x_1^2+x_2^4+x_3^5+x_4^4)^{6}\not \in (x_1^7,x_2^7,x_3^7,x_4^7)$ and therefore $R$ is strongly $F$-regular by Lemma~\ref{lem: criterion for strong F-regularity}. The chart $Y_1$ is non-singular. The charts $Y_2$ and $Y_4$ are isomorphic and have coordinate rings isomorphic to the hypersurface
\[
S=\frac{{\boldsymbol{k}}[a,b,c,d]}{(a^2+b^2+c^5b^3+d^4b^2)}.
\]
The hypersurface $S$ is not normal at the point $(a,b,c,d)$, in particular is not strongly F-regular, but is $F$-pure since $(a^2+b^2+c^5b^3+d^4b^2)^6\not\in (a^7,b^7,c^7,d^7)$. The remaining chart has coordinate ring isomorphic to
\[
\frac{{\boldsymbol{k}}[a,b,c,d]}{(a^2+b^2+c^5b^3+d^4b^2)},
\]
a ring which is neither normal nor F-pure.

Observe that $R_\m$ is  a local ring of multiplicity $2$. In particular $\e(R_\m)+s(R_\m)=2$, see the proof of \cite[Proposition~4.22]{TuckerFSigExists} for a justification. The same holds for the three singular charts of the blow-up. In particular, not only does the $F$-signature strictly decrease to $0$ on points in the exceptional locus of $\pi:Y\to X$, but the Hilbert-Kunz multiplicity of these points has strictly increased to $2$.

We leave it to the reader to verify that if $\widetilde{Y}\to Y$ is the normalization of $Y$, i.e., $\widetilde{Y}\to X$ is the normalized blow-up of $X$ at the origin, then $\widetilde{Y}$ is non-singular and in particular the conclusion of Theorem~\ref{thm:secondgoal} is valid for the proper birational morphism $\widetilde{Y}\to X$. This is not an indication that the conclusion of Theorem~\ref{thm:secondgoal} is valid for normalized blow-ups of isolated strongly $F$-regular singularities by the following examples.

\end{example}

\begin{example}\label{example: second example}

Let
\[
R=\frac{{\boldsymbol{k}}[x_1,x_2,x_3,x_4]}{(x_1^2+x_2^3+x_3^6+x_4^6)}
\]
where ${\boldsymbol{k}}$ is an algebraically closed field of characteristic $7$. Then $R$ is strongly $F$-regular but the affine chart $Y_4$ of the blow-up has coordinate ring isomorphic to
\[
\frac{{\boldsymbol{k}}[a,b,c,d]}{(a^2+b^3d+c^6d^4+d^4)}
\]
which is normal, $F$-pure, but is not strongly $F$-regular.

\end{example}

Our final example illustrates that the normalized blow-up of an isolated F-regular singularity can produce a normal variety with non-F-pure points. The example is obtained by changing the characteristic of the base field from Example~\ref{example: second example}.

\begin{example}\label{example: third example}
Let
\[
R=\frac{{\boldsymbol{k}}[x_1,x_2,x_3,x_4]}{(x_1^2+x_2^3+x_3^6+x_4^6)}
\]
where ${\boldsymbol{k}}$ is an algebraically closed field of characteristic $11$. Then $R$ is strongly $F$-regular but the affine chart $Y_4$ of the blow-up has coordinate ring isomorphic to
\[
\frac{{\boldsymbol{k}}[a,b,c,d]}{(a^2+b^3d+c^6d^4+d^4)}
\]
which is normal but not $F$-pure.

\end{example}

\section{Further questions}

We conclude the paper by stating two open questions.  First, we hope that Hilbert-Kunz multiplicity can also be controlled under certain blowups.

\begin{question}
Can we control the Hilbert-Kunz multiplicity of a local ring $(R, \fram)$ under (special) blowups $\pi : Y \to X = \Spec R$?
\end{question}

Second, we would like to generalize the results of \autoref{sec.FinitelyGeneratedCanonical} to the case when the ring $\bigoplus_{i \geq 0} R(iK_X)$ or $\bigoplus_{i \geq 0} R(-iK_X)$ is finitely generated, instead of generated in degree 1.  Note that we expect that for any strongly $F$-regular ring and any Weil divisor $D$, the ring $\bigoplus_{i \geq 0} R(iD)$ is finitely generated, this would hold for instance if the minimal model program is known to hold in characteristic $p > 0$ and hence we know it if $\dim R = 3$ and $p \geq 5$, see \cite[Theorem~1.3]{BirkarMMPCharp} and also \cite[Theorem~4.3]{SchwedeSmithLogFanoVsGloballyFRegular} applied locally.

\begin{question}
If $R$ is a strongly $F$-regular local ring and $S = \bigoplus_{i \geq 0} R(iK_X)$ (respectively, $S = \bigoplus_{i \geq 0} R(-iK_X)$) is finitely generated, can we control the $F$-signature of $\Proj S$?
\end{question}

\bibliographystyle{alpha}
\bibliography{MainBib}
\end{document}